\documentclass[11pt, amsfonts]{amsart}

\usepackage{amsmath,amssymb,amsthm,enumerate,comment}
\usepackage[all]{xy}
\usepackage{color}
\usepackage{graphicx}
\usepackage[T1]{fontenc}
\usepackage[dvipdfm,colorlinks=true]{hyperref}
\usepackage[colorlinks=true]{hyperref}

\SelectTips{eu}{12}

\theoremstyle{definition}
\newtheorem{theorem}{Theorem}[section]
\newtheorem{definition}[theorem]{\rm Definition}
\newtheorem{lemma}[theorem]{Lemma}
\newtheorem{proposition}[theorem]{Proposition}
\newtheorem{corollary}[theorem]{Corollary}

\newtheorem{remark}[theorem]{\rm Remark}

\makeatletter
\@addtoreset{equation}{section} 
\makeatother

\DeclareMathOperator{\Homeo}{Homeo}

\DeclareMathOperator{\id}{id}

\DeclareMathOperator{\Symp}{Symp}

\DeclareMathOperator{\Hom}{Hom}

\DeclareMathOperator{\grp}{grp}

\DeclareMathOperator{\ev}{ev}
\DeclareMathOperator{\Gau}{Gau}
\DeclareMathOperator{\Aut}{Aut}

\newcommand{\tHomeo}{\mathrm{H}\widetilde{\mathrm{omeo}}}
\newcommand{\HHH}{\mathrm{H}}
\newcommand{\ZZ}{\mathbb{Z}}
\newcommand{\RR}{\mathbb{R}}
\newcommand{\GG}{\Gamma}

\makeatletter
\@namedef{subjclassname@2020}{%
  \textup{2020} Mathematics Subject Classification}
\makeatother

\title[The Dixmier-Douady class and an abelian extension]{The Dixmier-Douady class and an abelian extension of the homeomorphism group}
\author{Shuhei Maruyama}
\address{Graduate School of Mathematics, Nagoya University, Japan}
\email{m17037h@math.nagoya-u.ac.jp}


\begin{document}

\begin{abstract}
  Let $X$ be a connected topological space and $c \in \HHH^2(X;\ZZ)$ a non-zero cohomology class.
  A $\Homeo(X,c)$-bundle is a fiber bundle with fiber $X$ whose structure group reduces to the group $\Homeo(X,c)$ of $c$-preserving homeomorphisms of $X$.
  If $\HHH^1(X;\ZZ) = 0$, then a characteristic class for $\Homeo(X,c)$-bundles called the Dixmier-Douady class is defined via the Serre spectral sequence.
  We show a relation between the universal Dixmier-Douady class for foliated $\Homeo(X,c)$-bundles and the gauge group extension of $\Homeo(X,c)$.
  Moreover, under some assumptions, we construct a central $S^1$-extension and a group two-cocycle on $\Homeo(X,c)$ corresponding to the Dixmier-Douady class.
\end{abstract}

\maketitle

\section{Introduction}

\subsection{Main results}

Let $X$ be a connected topological space whose homeomorphism group $\Homeo(X)$ is a topological group with respect to the compact-open topology.
Assume that there exists a non-zero cohomology class $c \in \HHH^2(X;\ZZ)$.
Let $G=\Homeo(X,c)$ be the group of $c$-preserving homeomorphisms of $X$, and $G^{\delta}$ the group $G$ with the discrete topology.
A fiber bundle $X \to E \to B$ is called a \emph{$G$-bundle} if the structure group reduces to $G$.
A $G$-bundle is said to be \emph{foliated} if the structure group reduces to the discrete group
$G^{\delta}$.

Assume that $\HHH^1(X;\ZZ) = 0$.
For a $G$-bundle $X \to E \to B$, we define a characteristic class $D_c(E)$ in $\HHH^3(B;\ZZ)$ as the image of $c$ under the transgression map 
(see Subsection \ref{construction}).
We call this class $D_c(E)$ the \emph{Dixmier-Douady class with respect to $c$}, or simply the \emph{Dixmier-Douady class} if $c$ is understood.
Let $D_c \in \HHH^3(BG;\ZZ)$ denote the universal Dixmier-Douady class.
Under the map
\begin{align}\label{discretizing_map_cohomology}
  \iota^*\colon \HHH^*(BG;\ZZ) \to \HHH^*(BG^{\delta};\ZZ)
\end{align}
induced from the canonical map $G^{\delta} \to G$ and the isomorphism $\HHH^3(BG^{\delta};\ZZ) \cong \HHH_{\grp}^3(G;\ZZ)$, we obtain the third group cohomology class $\iota^* D_c$ in $\HHH_{\grp}^3(G;\ZZ)$, which we call the \emph{discrete Dixmier-Douady class}.

In this paper we give algebraic interpretations of the discrete Dixmier-Douady class.
Throughout this paper, we regard the circle $S^1$ as the quotient group $\RR/\ZZ$.
Let $P \to X$ be a principal $S^1$-bundle whose first Chern class is equal to $c \in \HHH^2(X;\ZZ)$.
Then the group $\Aut(P)$ of bundle automorphisms of $P$ gives rise to a group extension
\begin{align}\label{gauge_extension}
  0 \to \Gau(P) \to \Aut(P) \to G \to 1.
\end{align}
Here $\Gau(P)$, which is called the \emph{gauge group of $P$}, is the kernel of the natural projection.
Since the group $S^1$ is abelian, the gauge group is also abelian.
Hence group extension (\ref{gauge_extension}) is an abelian extension.
In general, an abelian extension $1 \to A \to \Gamma \to G \to 1$ determines a second group cohomology class $e(\Gamma) \in \HHH_{\grp}^2(G; A)$.
Here $A$ is regarded as a right $G$-module by the conjugation $\GG$-action on $A$.
In particular, abelian extension (\ref{gauge_extension}) defines a group cohomology class $e(\Aut(P))$ in $\HHH_{\grp}^2(G;\Gau(P))$.

Here we obtain two group cohomology classes, the class $e(A_G(P)) \in \HHH_{\grp}^2(G;\Gau(P))$ corresponding to (\ref{gauge_extension}) and the discrete Dixmier-Douady class $\iota^*D_c \in \HHH_{\grp}^3(G;\ZZ)$.
Our main theorem (Theorem \ref{main_thm}) describes the relation between these two classes.

To state Theorem \ref{main_thm}, we need to define a map $\mathbf{d} \colon \HHH_{\grp}^2(G;\Gau(P)) \to \HHH_{\grp}^3(G;\ZZ)$.
Note that, since $S^1$ is abelian, the gauge group $\Gau(P)$ is isomorphic to the group $C(X,S^1)$ of $S^1$-valued continuous maps on $X$.
Let $C(X,\RR)$ denote the group of continuous maps from $X$ to $\RR$.
By the assumption $\HHH^1(X;\ZZ) = 0$ and the identification $\HHH^1(X;\ZZ) = \Hom(\pi_1(X),\ZZ) = \Hom(\pi_1(X), \pi_1(S^1))$, all elements in $C(X,S^1)$ can be lifted to elements in $C(X,\RR)$.
Hence the sequence
\begin{align}\label{coeff_ex_seq_gauge}
  0 \to \ZZ \to C(X,\RR) \to C(X,S^1) \to 0
\end{align}
is exact.
By (\ref{coeff_ex_seq_gauge}) and the identification $\Gau(P) \cong C(X, S^1)$, we have the connecting homomorphism
\begin{align*}
  \mathbf{d} \colon \HHH_{\grp}^2(G;\Gau(P)) \to \HHH_{\grp}^3(G;\ZZ).
\end{align*}

We are now ready to state Theorem \ref{main_thm}.

\begin{theorem}\label{main_thm}
  Let $\mathbf{d} \colon \HHH_{\grp}^2(G;\Gau(P)) \to \HHH_{\grp}^3(G;\ZZ)$ denote the connecting homomorphism.
  Then the following holds:
  \[
    \mathbf{d} (e(\Aut(P))) = \iota^*D_c.
  \]
\end{theorem}

From here to the end of this subsection, we further assume that the cohomology class $c \in \HHH^2(X;\ZZ)$ is equal to zero in $\HHH^2(X;\RR)$.
By abuse of notation, we use the symbol $\mathbf{d}$ to denote the connecting homomorphism $\mathbf{d} \colon \HHH_{\grp}^2(G;S^1) \to \HHH_{\grp}^3(G;\ZZ)$ induced from the exact sequence $0 \to \ZZ \to \RR \to S^1 \to 1$ of coefficients.
Since the sequence
\[
  \cdots \to \HHH^1(X;S^1) \overset{\mathbf{d}}{\to} \HHH^2(X;\ZZ) \to \HHH^2(X;\RR) \to \cdots
\]
is exact, there exists a cohomology class $\rho \in \HHH^1(X;S^1)$ satisfying $\mathbf{d}(\rho) = c$.
By using $\rho$, we can construct a central $S^1$-extension
\begin{align}
  1 \to S^1 \to \Aut(P_{\rho}^{\delta}) \to G \to 1
\end{align}
(see Section \ref{section:S^1-extension}).
This central $S^1$-extension determines a second cohomology class $e(\Aut(P_{\rho}^{\delta}))$ in $\HHH_{\grp}^2(G;S^1)$.
Then the following theorem holds, which is analogous to Theorem \ref{main_thm}.

\begin{theorem}\label{main_thm_S^1}
  Let $\mathbf{d} \colon \HHH_{\grp}^2(G;S^1) \to \HHH_{\grp}^3(G;\ZZ)$ be the connecting homomorphism with respect to the exact sequence $0 \to \ZZ \to \RR \to S^1 \to 1$ of coefficients.
  Then the following holds:
  \[
    \mathbf{d} (e(\Aut(P_{\rho}^{\delta}))) = \iota^*D_c.
  \]
\end{theorem}
Moreover, we describe a group two-cocycle $\mathfrak{G}_{x,\alpha} \in C_{\grp}^2(G;S^1)$ which represents the class $e(\Aut(P_{\rho}^{\delta}))$ (see (\ref{cocycle_S1_extension})).

\subsection{Background}

To give a context to our work, we briefly describe a related results.
Our main theorem can be seen as a higher dimensional analogue of the relation between the Euler class of foliated $S^1$-bundles and the universal covering of the group $\Homeo_+(S^1)$ of orientation preserving homeomorphisms of the circle.

One way to define the Euler class of (orientable) $S^1$-bundle $S^1 \to E \to B$ is to use the transgression map $d_2^{0,1} \colon \HHH^1(S^1;\ZZ) \to \HHH^2(B;\ZZ)$, that is, the Euler class $e(E)$ is defined by $e(E) = -d_2^{0,1}(c)$ for the generator $c \in \HHH^1(S^1;\ZZ)$.
The universal Euler class $e$ is the class in $\HHH^2(B\Homeo_+(S^1);\ZZ)$, where $\Homeo_+(S^1)$ is the group of orientation preserving homeomorphisms.
The universal Euler class of foliated $S^1$-bundles (discrete Euler class) is defined by $\iota^*e$, where $\iota^* \colon \HHH^2(B\Homeo_+(S^1);\ZZ) \to \HHH^2(B\Homeo_+(S^1)^{\delta};\ZZ) \cong \HHH_{\grp}^2(\Homeo_+(S^1);\ZZ)$ is the canonical map.
As an algebraic interpretation of the discrete Euler class, it is known that the central $\ZZ$-extension
\[
  0 \to \ZZ \to \tHomeo_+(S^1) \to \Homeo_+(S^1) \to 1
\]
induced from the universal covering of $\Homeo_+(S^1)$ corresponds to the universal Euler class of foliated $S^1$-bundles (see \cite{Fr} for example).

This correspondence between the Euler class of $S^1$-bundles and the central extension is generalized in \cite{1105.0825} as follows.
Let $X$ be a connected topological space and $c \in \HHH^1(X;\ZZ)$ a non-zero \emph{first} cohomology class.
Let $\Homeo(X,c)$ be the group of $c$-preserving homeomorphisms of $X$ and $X \to E \to B$ a $\Homeo(X,c)$-bundle.
A characteristic class $e_c(E) \in \HHH^2(B;\ZZ)$ of the $G$-bundle is defined as the image of $c$ under the transgression map.
The universal characteristic class $e_c$ is defined as an element of $H^2(B\Homeo(X,c);\ZZ)$ and the universal characteristic class of foliated $\Homeo(X,c)$-bundles is given by $\iota^*e_c$, where $\iota^*\colon \HHH^*(BG;\ZZ) \to \HHH^*(BG^{\delta};\ZZ)$ is the canonical map.
It was shown in \cite[Section 5]{1105.0825} that the class $\iota^*e_c$ corresponds to a central $\ZZ$-extension of $G$ given by the group of homeomorphisms of the covering space of $X$ associated to $c \in \HHH^1(X;\ZZ)$ which commute with the deck transformations.

Theorems \ref{main_thm} and \ref{main_thm_S^1} can be seen as results analogous to the above ones.
We note that in our setting, the class $e_c$ cannot be defined since there are no non-zero class in $\HHH^1(X;\ZZ)$.

There is also an analogous result in symplectic geometry (\cite{2009.01022}).
Let $(X,\omega)$ be a one-connected integral symplectic manifold, that is, the class $[\omega] \in \HHH^2(X;\RR)$ is in the image of the map $\HHH^2(X;\ZZ) \to \HHH^2(X;\RR)$).
Then, if we replace $\Homeo(X,c)$ with the \emph{symplectomorphism group} $\Symp(X,\omega)$, we obtain the Dixmier-Douady class for symplectic fibrations.
The Dixmier-Douady class for symplectic fibration (or Hamiltonian fibration) is studied in \cite{savelyev_shelukhin20}.
It is known that there exists a central $S^1$-extension of $\Symp(X,\omega)$ called the \emph{prequantum extension}, and this central extension corresponds to the discrete Dixmier-Douady class of foliated symplectic fibrations in the same manner as in Theorem \ref{main_thm_S^1} (\cite{2009.01022}).
Note that the existence of the prequantum extension is guaranteed without the assumption that the class $[\omega] \in \HHH^2(X;\RR)$ is zero, which corresponds to the condition in Theorem \ref{main_thm_S^1}.


\section{Preliminaries}\label{sec:pre}

\subsection{The Dixmier-Douady class}\label{construction}

The following description of the Dixmier-Douady class is based on \cite{carey_crowley_murray98}.
For the basics of the Dixmier-Douady class, see \cite{MR1197353} for example.

Let $X$ be a connected topological space satisfying $\HHH^1(X;\ZZ) = 0$.
Take a non-zero cohomology class $c$ in $\HHH^2(X;\ZZ)$.
Let $X \to E \to B$ denote a $G$-bundle with a connected base space $B$, and $(E_r^{p,q}, d_r^{p,q})$ the Serre spectral sequence of the bundle $X \to E \to B$.
By the assumption $\HHH^1(X;\ZZ) = 0$, we have
\[
  E_3^{0,2} = E_2^{0,2} = \HHH^0(B;\mathcal{H}^2(X;\ZZ)) 
\]
and
\[
  E_3^{3,0} = E_2^{3,0} = \HHH^3(B;\mathcal{H}^0(X;\ZZ)) =\HHH^3(B;\ZZ),
\]
where $\mathcal{H}^*(X;\ZZ)$ is the local system over $B$ induced from the bundle $X \to E \to B$.
It is known that the cohomology group $\HHH^0(B;\mathcal{H}^2(X;\ZZ))$ is isomorphic to the invariant part $\HHH^2(X;\ZZ)^{\pi_1(B)}$ of the monodromy action of $\pi_1(B)$ on $\HHH^2(X;\ZZ)$ (see \cite[Proposition 5.14]{MR1841974} for example).
Since the structure group of the bundle $E$ is contained in $G=\Homeo(X,c)$, the monodromy action preserves the class $c$, and hence the class $c$ can be seen as an element of $E_3^{0,2} = \HHH^0(B;\mathcal{H}^2(X;\ZZ))$.
By using the transgression map
\[
  d_3^{0,2} \colon E_3^{0,2} \to E_3^{3,0} = \HHH^3(B;\ZZ),
\]
we obtain the cohomology class $d_3^{0,2} c$ in $\HHH^3(B;\ZZ)$.
\begin{definition}
  For a $G$-bundle $X \to E \to B$, the cohomology class
  \[
    D_c(E) = -d_3^{0,2}c \in \HHH^3(B;\ZZ)
  \]
  is called the \emph{Dixmier-Douady class with respect to $c$} (or simply called the \emph{Dixmier-Douady class} if $c$ is understood).
\end{definition}

By the naturality of the Serre spectral sequence, the cohomology class $d_3^{0,2} c$ has also the naturality with respect to bundle maps.
Hence the class $D_c(E)$ gives rise to a characteristic class of $G$-bundles.
Let $D_c \in \HHH^3(BG;\ZZ)$ denote the universal Dixmier-Douady class of $G$-bundles.
By the canonical map $\iota^*\colon\HHH^3(BG;\ZZ) \to \HHH^3(BG^{\delta};\ZZ)$, we obtain the universal Dixmier-Douady class $\iota^*D_c$ of foliated $G$-bundles.

In section \ref{examples}, we give examples of $G$-bundles whose Dixmier-Douady class is non-zero.

\subsection{Group cohomology and Hochschild-Serre spectral sequence}
Let $G$ be a group and $A$ a right $G$-module.
A \emph{group $p$-cochain} on $G$ is a function from $p$-fold product $G^p$ to $A$.
Let $C_{\grp}^p(G;A)$ be the set of all group $p$-cochains, and the coboundary map $\delta \colon C_{\grp}^p(G;A) \to C_{\grp}^{p+1}(G;A)$ is defined by
\begin{align*}
  \delta c (g_1, \dots, g_{p+1}) = & c(g_2, \dots, g_{p+1}) + \sum_{i = 1}^{p}(-1)^i c(g_1, \dots, g_i g_{i+1}, \dots, g_{p+1}) \\
  & +(-1)^{p+1}c(g_1, \dots, g_p) \cdot g_{p+1}
\end{align*}
for $p > 0$ and $\delta = 0$ for $p = 0$.
The cohomology of the cochain complex $(C_{\grp}^*(G;A),\delta)$ is called the \emph{group cohomology of $G$ with coefficients in $A$} and denoted by $\HHH_{\grp}^*(G;A)$.
The group cohomology $\HHH_{\grp}^*(G;A)$ is known to be isomorphic to the singular cohomology $\HHH^*(BG^{\delta};\mathcal{A})$ of the classifying space $BG^{\delta}$, where $G^{\delta}$ denote the group $G$ with discrete topology and $\mathcal{A}$ is the local system on $BG^{\delta}$ induced from the $G$-module $A$ (see \cite{brown82}).

For an exact sequence $1 \to K \xrightarrow{i} \Gamma \xrightarrow{\pi} G \to 1$ of groups, there is a spectral sequence $(E_r^{p,q}, d_r^{p,q})$ with $E_2^{p,q} \cong \HHH_{\grp}^p(G;\HHH_{\grp}^q(K;A))$ which converges to $\HHH_{\grp}^*(\Gamma;A)$.
This spectral sequence is called the Hochschild-Serre spectral sequence (\cite{hochschild_serre53}).
Here, $\HHH_{\grp}^q(K;A)$ is seen as a right $G$-module as follows.
For a $q$-cochain $c$ on $K$, we define a $\Gamma$-action by
\[
  (c\cdot \gamma)(k_1, \cdots, k_q) = c(\gamma k_1 \gamma^{-1}, \cdots, \gamma k_q \gamma^{-1})\cdot \gamma.
\]
This action induces a $G$-action on $\HHH_{\grp}^q(K;A)$.
Let $\HHH_{\grp}^q(K;A)^G$ denote the $G$-invariant part.

The transgression map $d_r^{0, r-1} \colon E_r^{0, r-1} \to E_r^{0,r}$ is described as follows:
\begin{proposition}[{\cite[p.129]{hochschild_serre53}}]\label{lem:transgression_description}
  Let $f \in C_{\grp}^r(\GG;A)$ be a cochain.
  Assume that there exists a cocycle $c \in C_{\grp}^{r+1}(G;A)$ such that the pullback $\pi^* c$ is equal to $\delta f$.
  Then, the cochain $f$ defines an element $[f]_r$ of $E_r^{0, r-1}$ and the cohomology class $d_r^{0,r-1}([f]_r)$ is equal to $[\pi^* c]_r \in E_r^{r,0}$. 
  Moreover, if $E_r^{0,r-1} \cong E_2^{0,r-1} \cong \HHH_{\grp}^{r-1}(K;A)$ and $E_r^{r,0} \cong E_2^{r,0} \cong \HHH_{\grp}^r(G;A)$,
  then the cohomology classes of $\HHH_{\grp}^{r-1}(K;A)$ and $\HHH_{\grp}^r(G;A)$ corresponding to $[f]_r$ and $d_r^{0,r-1}([f]_r)$ are represented by $f|_{K^{r-1}}$ and $c$, respectively.
\end{proposition}

Proposition \ref{lem:transgression_description} for the low degree cases $d_2^{0,1}$ and $d_3^{0,2}$ is described as follows.
By definition, the cohomology $E_2^{0,1} = \HHH_{\grp}^0(G;\HHH_{\grp}^1(K;A))$ is isomorphic to $\HHH_{\grp}^1(K;A)^G$.

\begin{corollary}\label{prop:transgression}  
  Let $f \in C_{\grp}^1(\GG;A)$ be a cochain.
  Assume that there exists a cocycle $c \in C_{\grp}^2(G;A)$ satisfying $\delta f = \pi^* c$.
  Then, the pullback $i^* f \in C_{\grp}^1(K;A)$ is a cocycle representing an element of $\HHH_{\grp}^1(K;A)^{G}$ and
  \[
    d_2^{0,1}([i^* f]) = [c] \in \HHH_{\grp}^2(G;A).
  \]
\end{corollary}

If $\HHH_{\grp}^1(K;A)$ is trivial, then so is $E_2^{p, 1} = \HHH_{\grp}^p(G;\HHH_{\grp}^1(K;A))$ for any $p \geq 0$.
Hence the transgression map $d_3^{0,2}$ induces a map $d_3^{0,2} \colon \HHH_{\grp}^2(K;A)^{G} = E_{3}^{0,2} \to E_{3}^{3,0} = \HHH_{\grp}^{3}(G;A)$.
By Proposition \ref{lem:transgression_description}, we have the following. 
\begin{corollary}\label{cor:3_transgression}
  Assume that $\HHH_{\grp}^1(K;A)$ is trivial.
  Let $f \in C_{\grp}^2(\GG;A)$ be a cochain.
  Assume that there exists a cocycle $c \in C_{\grp}^3(G;A)$ satisfying $\delta f = \pi^* c$.
  Then, the pullback $i^* f \in C_{\grp}^2(K;A)$ is a cocycle representing an element of $\HHH_{\grp}^2(K;A)^{G}$ and
  \[
    d_3^{0,2}([i^* f]) = [c] \in \HHH_{\grp}^3(G;A).
  \]
\end{corollary}

\subsection{Abelian extension and second group cohomology}
An exact sequence $1 \to A \xrightarrow{i} \Gamma \xrightarrow{p} G \to 1$ of groups is called an \emph{abelian extension of} $G$ if the group $A$ is abelian.
The group $\Gamma$ acts on $A$ by conjugation, and this action induces a right $G$-action on $A$.
Thus we consider the abelian group $A$ as a right $G$-module.
It is known that the second group cohomology $\HHH_{\grp}^2(G;A)$ is isomorphic to the set of equivalence classes of abelian $A$-extensions of $G$ (see \cite{brown82}).

For an abelian $A$-extension $\Gamma$, the corresponding cohomology class $e(\Gamma)$ is defined as follows.
Take a section $s \colon G \to \Gamma$ of the projection $p \colon \Gamma \to G$.
For any $g, h \in G$, the value $s(g)s(h)s(gh)^{-1}$ is in $i(A) \cong A$.
We define a group two-cochain $c \in C_{\grp}^2(G;A)$ by
\begin{align}\label{euler_cocycle}
  c(g,h) = s(g)s(h)s(gh)^{-1}.
\end{align}
This cochain $c$ is a cocycle, and its cohomology class $[c]$ does not depend on the choice of sections.
This cocycle represents the class $e(\Gamma)$.

By the definition of the derivations of the Hochschild-Serre spectral sequence, the cohomology class $e(\Gamma)$ can be described as follows (see \cite[(2.4.4) Theorem]{MR2392026} for example).

\begin{lemma}\label{lemma:identity-Euler_class}
  Let $1 \to A \to \Gamma \xrightarrow{\pi} G \to 1$ be an abelian $A$-extension of $G$ and $(E_r^{p,q}, d_r^{p,q})$ the Hochschild-Serre spectral sequence of it.
  Then the cohomology class $e(\Gamma)$ is equal to the negative of $d_2^{0,1}(\id_A)$, where $d_2^{0,1} \colon \HHH_{\grp}^1(A;A)^G = E_2^{0,1} \to E_2^{2,0} = \HHH_{\grp}^2(G;A)$ is the derivation of the spectral sequence and $\HHH_{\grp}^1(A;A)^G$ is the set of $G$-equivariant homomorphisms.
\end{lemma}

\subsection{Cohomology of the gauge group}\label{subsection:cohomology}
By the fibration
\[
  0 \to \ZZ \to C(X,\RR) \overset{\pi}{\to} C(X,S^1) (\cong \Gau(P)) \to 0
\]
and the contractibility of $C(X,\RR)$, the gauge group $\Gau(P)$ is turned out to be an Eilenberg-MacLane space $K(\ZZ,1)$.
Hence the classifying space $B\Gau(\ZZ)$ is $K(\ZZ, 2)$, and therefore we have
\[
  \HHH^1(B\Gau(P);\ZZ) = 0 \ \text{ and } \  \HHH^2(B\Gau(P);\ZZ) = \ZZ.
\]

Next we show $\HHH_{\grp}^1(\Gau(P);\ZZ) = 0$.
Note that the first cohomology $\HHH_{\grp}^1(\Gau(P);\ZZ)$ is isomorphic to the group $\Hom(\Gau(P),\ZZ)$ of homomorphisms from $\Gau(P)$ to $\ZZ$.
Since $\Gau(P)$ is a divisible group, we have $\HHH_{\grp}^1(\Gau(P);\ZZ) = \Hom(\Gau(P),\ZZ) = 0$.

Let $i \colon \RR \to C(X,\RR)$ and $j \colon S^1 \to \Gau(P) = C(X,S^1)$ be the inclusions, and $\ev_x^{\RR} \colon C(X,\RR) \to \RR$ and $\ev_x \colon C(X,S^1) \to S^1$ denote the evaluation maps at $x \in X$.
Then we have the following commutative diagram of fibrations
\[
\xymatrix{
B\ZZ \ar[r] \ar@{=}[d] &B\RR \ar[r] \ar[d]^{i} &BS^1 \ar[d]^{j}\\
B\ZZ \ar[r] \ar@{=}[d] &BC(X,\RR) \ar[r] \ar[d]^{\ev_x^{\RR}} & B\Gau(P) \ar[d]^{\ev_{x}}\\
B\ZZ \ar[r] &B\RR \ar[r] &BS^1,
}
\]
where we use, by abuse of notation, the symbols $i,j,\ev_x^{\RR}$, and $\ev_x$ to denote the induced maps between the classifying spaces.
Since the composition of $j$ and $\ev_x$ is equal to the identity homomorphism, so is the composition of $\ev_x^* \colon \HHH^*(BS^1;\ZZ) \to \HHH^*(B\Gau(P);\ZZ)$ and $j^* \colon \HHH^*(B\Gau(P);\ZZ) \to \HHH^*(BS^1;\ZZ)$.
Hence, the first Chern class of $B\ZZ \to BC(X,\RR) \to B\Gau(P)$ is equal to $\ev_x^* c_1$ and non-zero, where $c_1 \in \HHH^2(BS^1;\ZZ)$ is the universal first Chern class.

Let $e(\RR) \in \HHH_{\grp}^2(S^1;\ZZ)$ be the group cohomology class corresponding to the central extension $0 \to \ZZ \to \RR \to S^1 \to 1$.
It is known that the class $c_1 \in \HHH^2(BS^1;\ZZ)$ corresponds to $e(\RR)$ under the canonical isomorphism
\[
  \iota^* \colon \HHH^2(BS^1;\ZZ) \to \HHH^2(B(S^1)^{\delta};\ZZ) = \HHH_{\grp}^2(S^1;\ZZ),
\]
that is, $\iota^*c_1 = e(\RR)$.

\section{Proofs}
Recall that $c \in \HHH^2(X;\ZZ)$ is a non-zero class and $G = \Homeo(X,c)$ is the group of $c$-preserving homeomorphisms of $X$.
\begin{lemma}\label{lemma:exactness_of_gauge_extension}
  Let $P \to X$ be a principal $S^1$-bundle with the first Chern class $c \in \HHH^2(X;\ZZ)$.
  Let $\Aut(P)$ be the group of bundle automorphisms of $P$.
  Then the canonical projection $p\colon\Aut(P) \to \Homeo(X)$ gives the surjection $\Aut(P) \to G$.
  In particular, we have the following exact sequence
  \begin{align}\label{gau_extension_of_G}
    0 \to \Gau(P) \to \Aut(P) \to G \to 1.
  \end{align}
\end{lemma}

\begin{remark}
  The exact sequence above gives rise to a Serre fibration with respect to the compact-open topology (see \cite{yamanoshita95}).
\end{remark}

\begin{proof}[Proof of lemma $\ref{lemma:exactness_of_gauge_extension}$]
  For a bundle automorphism $F \colon P \to P$, the induced homeomorphism $p(F)\colon X \to X$ preserves the class $c$ by the naturality of the Chern class.
  Hence the homeomorphism $p(F)$ is in $G$.
  For a homeomorphism $f \in G$, let $f^*P \to X$ denote the pullback bundle of $P$ by $f$.
  Since the first Chern class of $f^*P \to X$ is equal to $f^*c = c$ and the first Chern class completely determines the isomorphic class of principal $S^1$-bundles, the bundles $f^* P \to X$ is isomorphic to $P\to X$.
  Hence we have the following diagram
  \[
  \xymatrix{
  P \ar[r]^-{\cong} \ar[d] & f^*P \ar[r] \ar[d] & P \ar[d]\\
  X \ar@{=}[r]^-{\id_X} &X \ar[r]^{f} & X,
  }
  \]
  and this gives a bundle automorphism of $P$ that covers the homeomorphism $f$.
\end{proof}

\begin{remark}\label{rem:conj_action}
  The $G$-action on $C(X,S^1) = \Gau(P)$ induced from the $\Aut(P)$-conjugation action is given by the pullback, that is, $\lambda \cdot g = g^* \lambda = \lambda \circ g$ for $\lambda \in C(X,S^1)$ and $g \in G$.
\end{remark}

Let $(E_r^{'p,q}, d_r^{'p,q'})$ and $(E_r^{''p,q}, d_r^{''p,q})$ be the Hochschild-Serre spectral sequences of exact sequence (\ref{gau_extension_of_G}) with coefficients in $\Gau(P)$ and $\ZZ$, respectively.
Here $\ZZ$ is regarded as a trivial module.
Then
\[
  E_2^{'0,1} = \HHH_{\grp}^0(G;\HHH_{\grp}^1(\Gau(P);\Gau(P))) = \HHH_{\grp}^1(\Gau(P);\Gau(P))^{G},
\]
where $\HHH_{\grp}^1(\Gau(P);\Gau(P))^{G}$ is the group of $G$-equivariant homomorphisms on $\Gau(P)$.
Since $\HHH_{\grp}^1(\Gau(P);\ZZ)$ is trivial (Subsection \ref{subsection:cohomology}), so is $E_2^{''p,1} = \HHH_{\grp}^p(G;\HHH_{\grp}^1(\Gau(P);\ZZ))$ for all $p \geq 0$.
Hence $d_2^{''0,2} = 0$, and we have
\[
  E_3^{''0,2} = E_2^{''0,2} = \HHH_{\grp}^0(G;\HHH_{\grp}^2(\Gau(P);\ZZ)) = \HHH_{\grp}^2(\Gau(P);\ZZ)^{G},
\]
where $\HHH_{\grp}^2(\Gau(P);\ZZ)^{G}$ is the group of $G$-invariant cohomology classes.

\begin{lemma}\label{lemma:id_connecting_hom}
  The connecting homomorphism
  \[
    \mathbf{d} \colon \HHH_{\grp}^1(\Gau(P);\Gau(P)) \to \HHH_{\grp}^2(\Gau(P);\ZZ)
  \]
  with respect to exact sequence (\ref{coeff_ex_seq_gauge}) of coefficients induces the map
  \[
    \mathbf{d} \colon \HHH_{\grp}^1(\Gau(P);\Gau(P))^{G} \to \HHH_{\grp}^2(\Gau(P);\ZZ)^{G}.
  \]
\end{lemma}

\begin{proof}
  Let $[\phi]$ be an element of $\HHH_{\grp}^1(\Gau(P);\Gau(P))^{G}$.
  Note that the cocycle $\phi$ is a $G$-equivariant homomorphism, that is, the following holds
  \[
    \phi(\lambda \cdot g) = \phi(\lambda) \cdot g = \phi(\lambda) \circ g
  \]
  for any $\lambda \in \Gau(P)$ and $g \in G$.
  For a point $x$ in $X$, let $s_x(\lambda) \colon X \to \RR$ denote the lift of $\lambda \colon X \to S^1 = \RR/\ZZ$ satisfying $s_x(\lambda) (x) \in [0,1)$.
  This map $s_x$ defines a section
  \[
    s_x \colon (\Gau(P) =) C(X,S^1) \to C(X,\RR).
  \]
  Then, by the definition of the connecting homomorphism, a cocycle $c_x \in C_{\grp}^2(\Gau(P);\ZZ)$ representing $\mathbf{d} [\phi] \in \HHH_{\grp}^2(\Gau(P);\ZZ)$ is given by
  \begin{align*}
    c_x(\lambda_1,\lambda_2) &= \delta (s_x \circ \phi)(\lambda_1,\lambda_2)\\
    &= s_x (\phi(\lambda_2)) - s_x (\phi(\lambda_1 \lambda_2)) + s_x (\phi(\lambda_1))
  \end{align*}
  for $\lambda_1, \lambda_2 \in \Gau(P)$.
  For $g \in G$, we set $y = g(x)$.
  Note that $s_x(\lambda \circ g) = (s_{y}(\lambda))\circ g$ for any $\lambda \in \Gau(P)$.
  Since the $G$-action $\Gau(P)$ is given by pullback (Remark \ref{rem:conj_action}), we have
  \begin{align*}
    (c_x \cdot g)(\lambda_1,\lambda_2) &= c_x(\lambda_1 \circ g, \lambda_2 \circ g)\\
    &= s_x (\phi(\lambda_2 \circ g)) - s_x (\phi((\lambda_1 \lambda_2)\circ g)) + s_x (\phi(\lambda_1\circ g))\\
    &=s_x (\phi(\lambda_2)\circ g) - s_x (\phi(\lambda_1 \lambda_2)\circ g) +s_x (\phi(\lambda_1)\circ g)\\
    &=(s_y(\phi(\lambda_2)) - s_y(\phi(\lambda_1 \lambda_2)) + s_y(\phi(\lambda_1)))\circ g\\
    &=s_y(\phi(\lambda_2)) - s_y(\phi(\lambda_1 \lambda_2)) + s_y(\phi(\lambda_1))\\
    &= c_y(\lambda_1, \lambda_2).
  \end{align*}
  Since the image $\mathbf{d}[\phi]$ of the connecting homomorphism is independent of the choice of sections, we have
  \[
    (\mathbf{d} [\phi])\cdot g = [c_x \cdot g] = [c_y] = \mathbf{d} [\phi].
  \]
  Hence the cohomology class $\mathbf{d} [\phi]$ is in $\HHH_{\grp}^2(\Gau(P);\ZZ)^{G}$.
\end{proof}

\begin{remark}\label{remark:ev}
  Let $e(C(X,\RR)) \in \HHH_{\grp}^2(\Gau(P);\ZZ)$ be the class corresponding to the central extension $0 \to \ZZ \to C(X,\RR) \to \Gau(P) \to 0$.
  Then, by the direct computation, we can show that
  \[
    -\mathbf{d} (\id_{\Gau(P)}) = e(C(X,\RR)).
  \]
  By the commutative diagram of central extensions
  \[
  \xymatrix{
  \ZZ \ar[r] \ar@{=}[d] &C(X,\RR) \ar[r] \ar[d]^{\ev_x^{\RR}} & \Gau(P) \ar[d]^{\ev_{x}}\\
  \ZZ \ar[r] &\RR \ar[r] &S^1,
  }
  \]
  we have $e(C(X, \RR)) = \ev_x^*e(\RR) = \ev_x^* \iota^* c_1$.
\end{remark}

\begin{lemma}\label{lemma:commuting_diagram}
  The following diagram
  \begin{align}\label{comm_diag:group_coh_conn_hom_derivation}
    \xymatrix{
    \HHH_{\grp}^1(\Gau(P);\Gau(P))^{G} = E_2^{'0,1} \ar[r]^-{d_2^{'0,1}} \ar[d]^{\mathbf{d}} & E_2^{'2,0} = \HHH_{\grp}^2(G;\Gau(P)) \ar[d]^{\mathbf{d}} \\
     \HHH_{\grp}^2(\Gau(P);\ZZ)^{G} = E_3^{''0,2} \ar[r]^-{-d_3^{''0,2}} & E_3^{''3,0} = \HHH_{\grp}^3(G;\ZZ)
    }
  \end{align}
  commutes, where the vertical maps $\mathbf{d}$ are the connecting homomorphisms. 
\end{lemma}
\begin{proof}
  Let $[\phi]$ be an element of $\HHH_{\grp}^1(\Gau(P);\Gau(P))^{G}$ and $s\colon G \to \Aut(P)$ a section such that $s(\id_X) = \id_P$.
  First, we give a cocycle representing $\mathbf{d} d_2^{'0,1}[\phi]$.
  Let us define a cochain $\phi_s \colon \Aut(P) \to \Gau(P)$ by
  \[
    \phi_s(F) = \phi(F \circ (s(p(F)))^{-1}) \in \Gau(P),
  \]
  where $p \colon \Aut(P) \to G$ is the projection.
  Then, a cocycle $c \in C_{\grp}^3(G;\Gau(P))$ is defined by
  \[
    p^* c = \delta \phi_s.
  \]
  Indeed, $\delta \phi_s(F_1, F_2) \in \Gau(P)$ \emph{does} depend on $p(F_1)$ and $p(F_2) \in G$ since the restriction $\phi_s|_{\Gau(P)}$ is equal to a homomorphism $\phi$. 
  Hence, by Corollary \ref{prop:transgression}, the cocycle $c$ represents $d_2^{'0,1}[\phi]$, and a cocycle representing $\mathbf{d} d_2^{'0,1}[\phi]$ is given by
  \begin{align}\label{cocycle_deri_connecting_hom}
    \delta (s_x \circ c) \in C_{\grp}^3(G;\ZZ),
  \end{align}
  where $s_x \colon \Gau(P) \to C(X,\RR)$ is the section defined in the proof of Lemma \ref{lemma:id_connecting_hom}.

  Next, we give a cocycle representing $d_3^{''0,2}\mathbf{d} [\phi]$.
  A cocycle of $\mathbf{d} [\phi]$ is given by the coboundary $c_x = \delta (s_x \circ \phi) \in C_{\grp}^2(\Gau(P);\ZZ)$.
  We set
  \[
    c'_x = \delta (s_x \circ \phi_s) - p^*(s_x \circ c) \in C_{\grp}^2(\Aut(P);C(X,\RR)).
  \]
  Then, the cochain $c'_x$ is in $C_{\grp}^2(\Aut(P);\ZZ)$ since it is trivial under the projection $C(X,\RR) \to C(X, S^1)$, and the restriction of $c'_x$ on $\Gau(P)$ is equal to $c_x$.
  Moreover, since the coboundary $\delta c'_x$ is equal to $-p^*\delta(s_x \circ c)$, a cocycle of $d_3^{''0,2}\mathbf{d} [\phi]$ is given by
  \begin{align}\label{cocycle_connecting_hom_deri}
    - \delta (s_x \circ c) \in C_{\grp}^3(G;\ZZ)
  \end{align}
  by Corollary \ref{cor:3_transgression}.
  By (\ref{cocycle_deri_connecting_hom}) and (\ref{cocycle_connecting_hom_deri}),
  we have $\mathbf{d} d_2^{'0,1}[\phi] = -d_3^{''0,2} \mathbf{d} [\phi]$ and the lemma follows.
\end{proof}

\begin{lemma}
  Let $X \to E \to BG$ be the universal $G$-bundle and $B\Gau(P) \to B\Aut(P) \to BG$ the fibration of classifying spaces of exact sequence (\ref{gauge_extension}).
  Then there exists a commutative diagram
  \begin{align}\label{diagram:MtoBS^1}
    \xymatrix{
    X \ar[r] \ar[d]^{f} &E \ar[r] \ar[d]^{\phi} &BG \ar@{=}[d]\\
    B\Gau(P) \ar[r] &B\Aut(P) \ar[r] &BG,
    }
  \end{align}
  where the map $f \to B\Gau(P)$ is the composition of the classifying map $X \to BS^1$ and the map $BS^1 \to B\Gau(P)$ induced from the inclusion $S^1 \to \Gau(P)$.
\end{lemma}

\begin{proof}
  Let $P_G = P \times_{\Gau(P)} \Gau(P) \to X$ be the associated bundle.
  Then the bundle
  \[
    E = EG \times_{\Aut(P)} P_G \to BG
  \]
  gives one of the models of the universal $G$-bundle with fiber $X$.
  Since $EG \times P_G \to E$ is a principal $\Aut(P)$-bundle, there exists a bundle map to the universal principal $\Aut(P)$-bundle
  \[
  \xymatrix{
  EG \times P_G \ar[r]^-{\Psi} \ar[d] &E\Aut(P) \ar[d]\\
  E \ar[r]^-{\psi} & B\Aut(P).
  }
  \]
  Define a map $\Phi \colon EG \times P_G \to EG \times E\Aut(P)$ by $\Phi(a,p) = (a, \Psi(a,p))$.
  Then the map $\Phi$ gives a bundle map from $EG \times P_G \to E$ to $EG \times E\Aut(P) \to EG \times_{\Aut(P)} E\Aut(P) = B\Aut(P)$.
  Let $\phi \colon E \to B\Aut(P)$ be the classifying map that is covered by $\Phi$.
  Then the map $\phi$ covers the identity on $BG$.
  Moreover, the restriction $\phi|_{X} = f\colon X \to B\Gau(P)$ gives rise to the classifying map of the bundle $P_G \to X$.
\end{proof}

\begin{proof}[Proof of Theorem $\ref{main_thm}$]
  Let us consider the Serre spectral sequences $E_{r}^{p, q}$ and $E_{r}^{'''p,q}$ of the fibrations $X \to E \to BG$ and $B\Gau(P) \to B\Aut(P) \to BG$ respectively.
  Since $\HHH^1(B\Gau(P);\ZZ) = 0$ by Subsection \ref{subsection:cohomology}, we have $E_3^{'''0,2} = E_2^{'''0,2} = \HHH^2(B\Gau(P);\ZZ)$ and $E_3^{'''3,0} = E_2^{'''3,0} = \HHH^3(BG;\ZZ)$.
  By the naturality of the Serre spectral sequence with respect to commutative diagram (\ref{diagram:MtoBS^1}), we have the commutative diagram
  \[
  \xymatrix{
  \HHH^2(B\Gau(P);\ZZ) = E_3^{'''0,2} \ar[r]^-{d_3^{'''0,2}} \ar[d]^{f^*}& E_3^{'''3,0} = \HHH^3(BG;\ZZ) \ar@{=}[d]\\
  \HHH^2(X; \ZZ) = E_3^{0,2} \ar[r]^-{d_3^{0,2}} & E_3^{3,0} = \HHH^3(BG;\ZZ).
  }
  \]
  Then we have $f^* \ev_x^* c_1 = c \in \HHH^2(X;\ZZ)$ since $f\colon X \to B\Gau(P)$ is the composite of the classifying map $X \to BS^1$ and the map $j \colon BS^1 \to B\Gau(P)$ (Subsection \ref{subsection:cohomology}).
  Hence, the universal characteristic class $D_c = d_3^{0,2}(c) \in \HHH^3(BG;\ZZ)$ is equal to $-d_3^{'''0,2} \ev_x^* c_1$.
  Let $E_r^{''p,q}$ denote the Serre spectral sequence of the fibration $B\Gau(P)^{\delta} \to B\Aut(P)^{\delta} \to BG^{\delta}$ (or, equivalently, the Hochschild-Serre spectral sequence of $1 \to \Gau(P) \to \Aut(P) \to G \to 1$).
  By the commutative diagram
  \[
  \xymatrix{
  B\Gau(P)^{\delta} \ar[r] \ar[d] & B\Aut(P)^{\delta} \ar[r] \ar[d] & BG^{\delta} \ar[d]\\
  B\Gau(P) \ar[r] & B\Aut(P) \ar[r] & BG
  }
  \]
  and the naturality, we have the commutative diagram
  \[
  \xymatrix{
  \HHH^2(B\Gau(P);\ZZ) \ar[r]^-{d_3^{'''0,2}} \ar[d]^{\iota^*}& \HHH^3(BG;\ZZ) \ar[d]^{\iota^*}\\
  \HHH_{\grp}^2(\Gau(P); \ZZ) \ar[r]^-{d_3^{''0,2}} &\HHH_{\grp}^3(G;\ZZ),
  }
  \]
  where we identify the group cohomology with the singular cohomology of classifying space of discrete groups.
  Then we have
  \begin{align*}
    \iota^* D_c = -\iota^* d_3^{'''0,2} \ev_x^* c_1 = -d_3^{''0,2} \iota^* \ev_x^* c_1 = -d_3^{''0,2} \ev_x^* \iota^* c_1.
  \end{align*}
  By commutative diagram (\ref{comm_diag:group_coh_conn_hom_derivation}) and lemma \ref{lemma:identity-Euler_class}, we have
  \begin{align*}
    \mathbf{d} e(\Aut(P)) = -\mathbf{d} d_2^{'0,1}(\id_{\Gau(P)}) = d_3^{''0,2} \mathbf{d} (\id_{\Gau(P)}).
  \end{align*}
  Since the class $\ev_x^*\iota^*c_1$ is equal to $-\mathbf{d}(\id_{\Gau(P)})$ by Remark \ref{remark:ev}, we have
  \begin{align*}
    \iota^*D_c = \mathbf{d} e(A_G(P)).
  \end{align*}
\end{proof}

\section{Central $S^1$-extension of $G$}\label{section:S^1-extension}
In this section, we assume that the connected topological space $X$ admits the universal covering space $\widetilde{X}$.
We take a non-zero cohomology class $c \in \HHH^2(X;\ZZ)$ which is equal to zero in $\HHH^2(X;\RR)$.

\subsection{Construction of a central $S^1$-extension}
By the cohomology long exact sequence
\[
  \cdots \to \HHH^1(X;\RR) \to \HHH^1(X;S^1) \overset{\mathbf{d}}{\to} \HHH^2(X;\ZZ) \to \HHH^2(X;\RR) \to \cdots,
\]
there exists an element $\rho \in \HHH^1(X;S^1)$ such that $\mathbf{d} (\rho) = c$.
By the isomorphism $\HHH^1(X;S^1) \cong \Hom(\pi_1(X);S^1)$, we regard the class $\rho$ as a homomorphism $\rho \colon \pi_1(X) \to S^1$.
We set $P_{\rho}^{\delta} = \widetilde{X}\times_{\rho} (S^1)^{\delta}$.
Then $P_{\rho}^{\delta} \to X$ is a foliated $S^1$-bundle with the holonomy homomorphism $\rho$.
Note that the group $\Homeo(X,\rho)$ of $\rho$-preserving homeomorphisms of $X$ is equal to $G$ since $\mathbf{d}$ is injective.
\begin{lemma}
  Let $\Aut(P_{\rho}^{\delta})$ denote the bundle automorphisms of $P_{\rho}^{\delta}$.
  Then the canonical projection $\pi\colon\Aut(P_{\rho}^{\delta}) \to \Homeo(X)$ induces a surjection
  $\Aut(P_{\rho}^{\delta}) \to G$ and its kernel is isomorphic to $S^1$.
  In particular, we obtain an exact sequence
  \begin{align}\label{cent_S1_extension}
    1 \to S^1 \to \Aut(P_{\rho}^{\delta}) \to G \to 1.
  \end{align}
\end{lemma}
\begin{proof}
  We regard $G$ as the group of $\rho$-preserving homeomorphisms of $X$.
  First, we show that the image $\pi(\Aut(P_{\rho}^{\delta}))$ is contained in $G$.
  Take a bundle automorphism $F \in \Aut(P_{\rho}^{\delta})$ that covers a homeomorphism $f \colon X \to X$.
  Let $\gamma \colon [0,1] \to X$ be a loop and $\widetilde{\gamma}\colon[0,1] \to P_{\rho}^{\delta}$ a lift of $\gamma$.
  Then $F \circ \widetilde{\gamma}\colon [0,1] \to P_{\rho}^{\delta}$ is a lift of the loop $f \circ  \gamma$.
  Since $\rho \colon \pi_1(X) \to (S^1)^{\delta}$ is the holonomy of $P^{\delta}$, the value $\rho(f \circ \gamma) \in (S^1)^{\delta}$ is given by
  \begin{align}\label{holonomy_calc}
    F(\widetilde{\gamma}(1)) = F(\widetilde{\gamma}(0))\cdot \rho(f \circ \gamma).
  \end{align}
  Since the left-hand side of (\ref{holonomy_calc}) is equal to $F(\widetilde{\gamma}(0)\cdot \rho(\gamma)) = F(\widetilde{\gamma}(0))\cdot \rho(\gamma)$, we have $\rho(\gamma) = \rho(f \circ \gamma) = f^*\rho (\gamma)$.
  Hence the homeomorphism $f$ preserves the holonomy $\rho$ and we have $\pi(\Aut(P_{\rho}^{\delta})) \subset G$.

  Next we show that the map $\pi\colon\Aut(P_{\rho}^{\delta}) \to G$ is surjective.
  For a homeomorphism $f \colon X \to X$ that preserves $\rho$, we take a lift $\widetilde{f} \colon \widetilde{X} \to \widetilde{X}$.
  Then the map $\widetilde{X}\times (S^1)^{\delta} \to \widetilde{X} \times (S^1)^{\delta}$ defined by $(x,u) \to (\widetilde{f}(x),u)$ induces a bundle automorphism $F\colon P_{\rho}^{\delta} \to P_{\rho}^{\delta}$ that covers $f$.
  Hence the map $\pi\colon\Aut(P_{\rho}^{\delta}) \to G$ is surjective.

  The kernel of $\pi\colon\Aut(P_{\rho}^{\delta}) \to G$ is the gauge group $\Gau(P_{\rho}^{\delta})$ and this is isomorphic to $C(X,(S^1)^{\delta}) \cong (S^1)^{\delta}$ since the fiber of $P_{\rho}^{\delta}$ is the discrete group $(S^1)^{\delta}$ and $X$ is connected.
\end{proof}

Note that the group $S^1$ is in the center of $\Aut(P_{\rho}^{\delta})$, that is, (\ref{cent_S1_extension}) is a central $S^1$-extension of $G$.
Let $e(\Aut(P_{\rho}^{\delta})) \in \HHH_{\grp}^2(G;S^1)$ be the cohomology class corresponding to (\ref{cent_S1_extension}).
Let $P_{\rho} \to X$ be a principal $S^1$-bundle defined by $P_{\rho} = \widetilde{X} \times_{\rho} S^1$.
Then we have the following commutative diagram of groups
\begin{align}
  \xymatrix{
  1 \ar[r] & S^1 \ar[r] \ar[d]^{j} &\Aut(P_{\rho}^{\delta}) \ar[r] \ar[d]^{k} & G \ar[r] \ar@{=}[d] & 1\\
  1 \ar[r] & \Gau(P_{\rho}) \ar[r] & \Aut(P_{\rho}) \ar[r] & G \ar[r] & 1,
  }
\end{align}
where $j \colon S^1 \to \Gau(P_{\rho})$ is the inclusion.
Let $e(\Aut(P_{\rho})) \in \HHH_{\grp}^2(G;\Gau(P_{\rho}))$ be the cohomology class corresponding to the abelian extension $\Aut(P_{\rho})$.
\begin{lemma}\label{lem:j_*_coeff_change}
  Let $j_* \colon \HHH_{\grp}^2(G;S^1) \to \HHH_{\grp}^2(G;\Gau(P_{\rho}))$ be the change of coefficients homomorphism induced by $j$.
  Then we have $j_*(e(\Aut(P_{\rho}^{\delta}))) = e(\Aut(P_{\rho}))$ in $\HHH_{\grp}^2(G;\Gau(P_{\rho}))$.
\end{lemma}
\begin{proof}
  Recall that a cocycle $c$ representing $e(\Aut(P_{\rho}^{\delta}))$ is given by
  \[
    c(g,h) = s(g)s(h)s(gh)^{-1},
  \]
  where $s\colon G \to \Aut(P_{\rho}^{\delta})$ is a section.
  Since $k \circ s \colon G \to \Aut(P_{\rho})$ is a section of $\Aut(P_{\rho}) \to G$, a cocycle $c'$ representing $e(\Aut(P_{\rho}))$ is given by
  \[
    c'(g,h) = k(s(g))k(s(h))k(s(gh))^{-1} = k(s(g)s(h)s(gh)^{-1}) = j(c(g,h)).
  \]
  Hence we have $j_*(e(\Aut(P_{\rho}^{\delta}))) = e(\Aut(P_{\rho}))$.
\end{proof}

\begin{proof}[Proof of Theorem $\ref{main_thm_S^1}$]
  By the commutative diagram of coefficients
  \begin{align*}
    \xymatrix{
    0 \ar[r]& \ZZ \ar[r] \ar[d] & \RR \ar[r] \ar[d] & S^1 \ar[r] \ar[d]^{j} & 1 \\
    0 \ar[r] & \ZZ \ar[r] & C(X,\RR) \ar[r] & C(X,S^1) = \Gau(P_{\rho}) \ar[r] & 1,
    }
  \end{align*}
  we have the commutative diagram
  \begin{align*}
    \xymatrix{
    \HHH_{\grp}^2(G;S^1) \ar[r]^{\mathbf{d}} \ar[d]^{j^*} & \HHH_{\grp}^3(G;\ZZ) \ar@{=}[d] \\
    \HHH_{\grp}^2(G;\Gau(P_{\rho})) \ar[r]^-{\mathbf{d}} & \HHH_{\grp}^3(G;\ZZ),
    }
  \end{align*}
  where the maps $\mathbf{d}$ are the connecting homomorphisms.
  Together with Theorem \ref{main_thm} and Lemma \ref{lem:j_*_coeff_change}, we obtain
  \[
    \mathbf{d} e(A_G(P_{\rho}^{\delta})) = \mathbf{d} j_* e(A_G(P_{\rho}^{\delta})) = \mathbf{d} e(A_G(P_{\rho})) = \iota^* D_c.
  \]
\end{proof}

\subsection{Cocycle description of $e(\Aut(P_{\rho}^{\delta}))$}
In this subsection, we give a geometric description of a group two-cocycle representing the class $e(\Aut(P_{\rho}^{\delta}))$.
Let $c \in \HHH^2(X;\ZZ)$ and $\rho \in \HHH^1(X;S^1)$ be a cohomology class as above.
Let $\alpha \in C^1(X;S^1)$ be a cocycle representing the cohomology class $\rho$, where $(C^*(X;S^1),d)$ denotes the singular cochain complex of $X$ with coefficients in $S^1$.
The singular cochain $C^*(X;S^1)$ and the cohomology $\HHH^*(X;S^1)$ are the right $G$-modules by pullback.
Since $g^*\rho = \rho$ for any homeomorphism $g \in G$, the cochain $g^*\alpha - \alpha$ is a coboundary.
Hence there exists a cochain $\mathfrak{K}_{\alpha}(g) \in C^0(X;S^1)$ satisfying
\[
  g^*\alpha - \alpha = d\mathfrak{K}_{\alpha}(g).
\]
Then a two-cochain $\mathfrak{G}_{x,\alpha}$ in $C^2(G;S^1)$ is defined by
\begin{align}\label{cocycle_S1_extension}
  \mathfrak{G}_{x,\alpha}(g,h) = \int_{x}^{hx} g^* \alpha - \alpha = \mathfrak{K}_{\alpha}(g)(hx) - \mathfrak{K}_{\alpha}(g)(x),
\end{align}
where $x \in X$.
Here the symbol $\int_x^{hx}$ denotes the pairing of the cocycle $g^*\alpha - \alpha$ and a path from $x$ to $hx$.
By the arguments same as in \cite[Theorem 3,1]{ismagilov_losik_michor06}, we have the following proposition.

\begin{proposition}
  The two-cochain $\mathfrak{G}_{x,\alpha} \in C_{\grp}^2(G;S^1)$ is a cocycle and the cohomology class $[\mathfrak{G}_{x,\alpha}]$ is independent of the choice of the point $x \in X$ and the singular cochain $\alpha$.
\end{proposition}

Then the following theorem holds.
\begin{theorem}\label{thm:cocycle_description}
  The cohomology class $[\mathfrak{G}_{x,\alpha}]$ is equal to $e(\Aut(P_{\rho})) \in \HHH_{\grp}^2(G;S^1)$.
\end{theorem}

For the proof of Theorem \ref{thm:cocycle_description}, we need the following lemmas.
Let us recall that $p \colon P_{\rho}^{\delta} \to X$ is the principal $(S^1)^{\delta}$-bundle with the holonomy $\rho$.

\begin{lemma}\label{lemma:pullback_of_rho}
  The pullback $p^* \rho \in \HHH^1(P_{\rho}^{\delta};S^1)$ is equal to zero.
\end{lemma}
\begin{proof}
  Let $\gamma \colon [0,1] \to P_{\rho}^{\delta}$ be a based loop.
  Then $\gamma$ is a lift of the loop $p\circ\gamma$ in $X$.
  Since $\rho \colon \pi(X) \to S^1$ is the holonomy homomorphism, we have
  \[
    p^* \rho(\gamma) = \rho(p\circ \gamma) = \gamma(1) - \gamma(0) = 0,
  \]
  and the lemma follows.
\end{proof}

By lemma \ref{lemma:pullback_of_rho}, we take a singular cochain $\theta' \in C^0(P_{\rho}^{\delta};S^1)$ satisfying $d\theta' = p^* \alpha$.
We take a base point $y \in P_{\rho}^{\delta}$ and set $p(y)=x$.
For any $u \in S^1$, let $P_{\rho}^{\delta}(u)$ denote the connected component of $P_{\rho}^{\delta}$ which contains the point $y\cdot u$.
Then we set
\[
  z(u) =\theta'(y) - \theta'(y \cdot u) + u \in S^1.
\]
By a straightforward calculation, we have the equality $z(u) = z(u + \rho(\gamma))$ for any $\gamma \in \pi_1(X)$.
Hence, the value $z(u)$ gives a continuous function $z \colon P_{\rho}^{\delta} \to (S^1)^{\delta}$.
Note that $z(y\cdot u) = z(u)$ and $z(y) = 0$.
We set $\theta = \theta' + z \in C^0(X;\ZZ)$.
Since $dz = 0$, we have $d\theta = p^*\alpha$.
Let $\tau \in C_{\grp}^1(A_G(P_{\rho}^{\delta});S^1)$ be a cochain defined by
\[
  \tau(\phi) = \int_{y}^{\phi y} \theta = \theta(\phi y) - \theta(y).
\]

\begin{lemma}\label{lemma:connection_cochain}
  The restriction $\tau|_{S^1} \colon S^1 \to S^1$ is equal to the identity.
\end{lemma}

\begin{proof}
  By the definition of $\tau$, we have
  \begin{align*}
    \tau(u) = \theta(y\cdot u) - \theta(y) =\theta'(y\cdot u) + z(y\cdot u) -\theta'(y) - z(y)= u
  \end{align*}
  for any $u \in S^1$.
\end{proof}

\begin{lemma}\label{lemma:curvature}
  The equality
  \[
    -\delta \tau = \pi^* \mathfrak{G}_{x,\alpha} \in C_{\grp}^2(\Aut(P_{\rho}^{\delta});S^1)
  \]
  holds, where $\pi \colon \Aut(P_{\rho}^{\delta}) \to G$ is the projection.
\end{lemma}

\begin{proof}
  For the pullback $p^*\mathfrak{K}_{\alpha}(g) \in C^0(P_{\rho}^{\delta};S^1)$, we have
  \begin{align*}
    d(p^*\mathfrak{K}_{\alpha}(g)) &= p^*(d\mathfrak{K}_{\alpha}(g)) =p^*g^*\alpha - p^* \alpha\\
    &= \phi^* p^* \alpha - p^* \alpha =\phi^* d\theta -d\theta = d(\phi^* \theta - \theta),
  \end{align*}
  where $\phi \in \Aut(P_{\rho}^{\delta})$ is a lift of $g \in G$.
  Hence, for $\phi,\psi \in \Aut(P_{\rho}^{\delta})$ which cover $g, h \in G$, respectively, we have
  \begin{align*}
    \pi^*\mathfrak{G}_{x,\alpha}(\phi, \psi) &= \mathfrak{G}_{p(y),\alpha}(g,h) = \mathfrak{K}_{\alpha}(g)(hp(y)) - \mathfrak{K}_{\alpha}(g)(p(y))\\
    &=\mathfrak{K}_{\alpha}(g)(p\psi (y)) - \mathfrak{K}_{\alpha}(g)(p(y))\\
    &=p^*\mathfrak{K}_{\alpha}(g)(\psi(y)) - p^*\mathfrak{K}_{\alpha}(g)(y)\\
    &=(\phi^*\theta - \theta)(\psi(y)) - (\phi^*\theta - \theta)(y)\\
    &=(\theta(\phi \psi(y)) - \theta(y)) -(\theta(\phi(y)) - \theta(y)) -(\theta(\psi(y)) - \theta(y))\\
    &= -\delta \tau(\phi,\psi),
  \end{align*}
  and the lemma follows.
\end{proof}

\begin{proof}[Proof of theorem $\ref{thm:cocycle_description}$]
  Let $E_r^{p,q}$ denote the Hochschild-Serre spectral sequence of $0 \to S^1 \to \Aut(P_{\rho}^{\delta}) \to G \to 1$ with coefficients in $S^1$.
  Then the derivation $d_2^{0,1}$ defines a map
  \[
    d_2^{0,1} \colon \HHH_{\grp}^1(S^1;S^1) = E_2^{0,1} \to E_2^{2,0} = \HHH_{\grp}^2(G;S^1).
  \]
  By Corollary \ref{prop:transgression}, Lemma \ref{lemma:connection_cochain}, and Lemma \ref{lemma:curvature}, we have $d_2^{0,1} (\id_{S^1}) = -[\mathfrak{G}_{x,\alpha}]$.
  On the other hand, by Lemma \ref{lemma:identity-Euler_class}, we have $d_2^{0,1}(\id_{S^1}) = -e(\Aut(P_{\rho}^{\delta}))$.
  Hence we have $[\mathfrak{G}_{x,\alpha}] = e(\Aut(P_{\rho}^{\delta}))$.
\end{proof}

\section{Examples}\label{examples}
In this section, we give two examples that the universal characteristic classes $D_c$ and $\iota^*D_c$ are non-trivial.
The first example is the complex projective space $\mathbb{C}P^n$.
The proof of the following theorem is similar to the proof of \cite[Theorem 7.1]{2009.01022} (and also similar to the proof of Theorem \ref{thm:non-trivial_real} below), so we will omit it.
\begin{theorem}\label{thm:non-trivial_complex}
  Let $n$ be a positive integer and $c \in \HHH^2(\mathbb{C}P^n;\ZZ)$ the generator.
  Let $G$ denote the group $\Homeo(\mathbb{C}P^n, c)$ of $c$-preserving homeomorphisms.
  Then the universal characteristic classes $D_c \in \HHH^3(BG;\ZZ)$ and $\iota^* D_c \in \HHH^3(BG^{\delta};\ZZ)$ are non-zero.
\end{theorem}

The cohomology class $c$ in Theorem \ref{thm:non-trivial_complex} is non-zero in $\HHH^2(\mathbb{C}P^n;\RR)$.
Hence this example does not satisfy the assumption in Section \ref{section:S^1-extension}.
An example which satisfies the assumption in Section \ref{section:S^1-extension} is given as follows.

\begin{theorem}\label{thm:non-trivial_real}
  Let $c$ be the non-zero element in $\HHH^2(\RR P^3;\ZZ)\cong \ZZ/2\ZZ$.
  Let $G$ denote the group $\Homeo(\RR P^3,c)$ of $c$-preserving homeomorphisms.
  Then the universal characteristic classes $D_c \in \HHH^3(BG;\ZZ)$ and $\iota^*D_c \in \HHH^3(BG^{\delta};\ZZ)$ are non-zero.
\end{theorem}

\begin{proof}
  Since $\RR P^3$ is homeomorphic to $SO(3)$, the group $SO(3)$ is included in $G$ as the left translations.
  Hence we have the commutative diagram
  \begin{align*}
    \xymatrix{
    SO(3) \ar[r] \ar[d] & ESO(3) \ar[r] \ar[d]^{F} & BSO(3) \ar[d]^{f}\\
    G \ar[r] & EG \ar[r] & BG.
    }
  \end{align*}
  Recall that a universal $G$-bundle $\RR P^3 \to E \to BG$ is given as $E = EG \times_G \RR P^3$.
  For the base point $b \in \RR P^3$ corresponding to the unit $1 \in SO(3)$, we define a map $\phi \colon ESO(3) \to E$ by $\Phi(x) = [F(x),b]$ for $x \in ESO(3)$.
  Then we have the commutative diagram
  \begin{align}\label{diagram_SO(3)_HomeoRP3}
    \xymatrix{
    SO(3) \ar[r] \ar@{=}[d] & ESO(3) \ar[r] \ar[d]^{\Phi} & BSO(3) \ar[d]^{f}\\
    SO(3) \cong \RR P^3 \ar[r] & E \ar[r] & BG.
    }
  \end{align}
  Let us consider the Serre spectral sequences of the fibrations in (\ref{diagram_SO(3)_HomeoRP3}).
  By the naturality of the Serre spectral sequence, the pullback $f^*D_c \in \HHH^3(BSO(3);\ZZ)$ is equal to the transgression image of $c\in \HHH^2(\RR P^3;\ZZ) \cong \HHH^2(SO(3);\ZZ)$ with respect to the fibration $SO(3) \to ESO(3) \to BSO(3)$.
  Since $ESO(3)$ is contractible, the transgression map is injective.
  Hence the class $f^*D_c$ is non-zero, and so is $D_c$.

  Let us consider the commutative diagram
  \begin{align*}
    \xymatrix{
    \HHH^3(BG;\ZZ) \ar[r]^{\iota^*} \ar[d]^{f^*} & \HHH^3(BG^{\delta};\ZZ) \ar[d] \\
    \HHH^3(BSO(3);\ZZ) \ar[r]^{\iota^*} & \HHH^3(BSO(3)^{\delta};\ZZ).
    }
  \end{align*}
  Since the map $\iota^*\colon\HHH^3(BSO(3);\ZZ) \to \HHH^3(BSO(3)^{\delta};\ZZ)$ is injective \cite{milnor83}, the class $\iota^*f^*D_c \in \HHH^3(BSO(3)^{\delta};\ZZ)$ is non-zero.
  Hence the class $\iota^*D_c \in \HHH^3(BG^{\delta};\ZZ)$ is also non-zero.
\end{proof}

\begin{remark}
  By Theorem \ref{main_thm_S^1} and Theorem \ref{thm:non-trivial_real}, the class $e(\Aut(P_{\rho}^{\delta}))$ is also non-zero for $G = \Homeo(\RR P^3,c)$.
  Hence, this gives an example that cocycle (\ref{cocycle_S1_extension}) is cohomologically non-trivial.
\end{remark}

\begin{remark}
  The non-triviality of the class $c$ does not imply the non-triviality of $D_c$ nor $\iota^*D_c$.
  For example, we consider the real projective plane $\RR P^2$.
  Let $c \in \HHH^2(\RR P^2;\ZZ)\cong \ZZ/2\ZZ$ be a non-zero element and set $G = \Homeo(\RR P^2, c)$.
  Then the class $D_c \in \HHH^3(BG;\ZZ)$ is equal to zero.
  To see this, we consider the inclusion $i \colon SO(3) \hookrightarrow \Homeo_0(\RR P^2)$ induced from the natural $SO(3)$-action on $\RR P^2$.
  This inclusion is a homotopy equivalence (\cite[Corollary 1.2]{2108.02134}), and hence the map $i^* \HHH^3(BG;\ZZ) \to \HHH^3(BG;\ZZ)$ is an isomorphism.
  Let us consider the commutative diagram of fibrations
  \begin{align}\label{diagram_real_proj}
    \xymatrix{
    SO(3) \ar[r] \ar[d]^-{f} & ESO(3) \ar[r] \ar[d] & BSO(3) \ar[d] \\
    \RR P^2 \ar[r] & EG \times_{G} \RR P^2 \ar[r] & BG.
    }
  \end{align}
  Since $f \colon SO(3) \to \RR P^2$ is induced from the natural $SO(3)$-action on $S^2$, the map $f$ factors through $S^2$.
  Since $\HHH^2(SO(3);\ZZ) \cong \HHH^2(\RR P^2;\ZZ) \cong \ZZ/2\ZZ$ and $\HHH^2(S^2;\ZZ) \cong \ZZ$, the map $f$ induces the zero map on second cohomology.
  The naturality of the Dixmier-Douady class with respect to (\ref{diagram_real_proj}), together with the fact that the induced map $i^* \colon \HHH^3(BG;\ZZ) \to \HHH^3(BSO(3);\ZZ)$ is an isomorphism, implies that the Dixmier-Douady class $D_c$ is equal to zero.

\end{remark}

\bibliographystyle{amsalpha}
\bibliography{gauge.bib}

\providecommand{\bysame}{\leavevmode\hbox to3em{\hrulefill}\thinspace}
\providecommand{\MR}{\relax\ifhmode\unskip\space\fi MR }
\providecommand{\MRhref}[2]{%
  \href{http://www.ams.org/mathscinet-getitem?mr=#1}{#2}
}
\providecommand{\href}[2]{#2}
\begin{thebibliography}{CCM98}

\bibitem[Bro82]{brown82}
Kenneth~S. Brown, \emph{Cohomology of groups}, Graduate Texts in Mathematics,
  vol.~87, Springer-Verlag, New York-Berlin, 1982.

\bibitem[Bry93]{MR1197353}
Jean-Luc Brylinski, \emph{Loop spaces, characteristic classes and geometric
  quantization}, Progress in Mathematics, vol. 107, Birkh\"{a}user Boston,
  Inc., Boston, MA, 1993.

\bibitem[CCM98]{carey_crowley_murray98}
Alan~L. Carey, Diarmuid Crowley, and Michael~K. Murray, \emph{Principal bundles
  and the {D}ixmier {D}ouady class}, Comm. Math. Phys. \textbf{193} (1998),
  no.~1, 171--196.

\bibitem[DK01]{MR1841974}
James~F. Davis and Paul Kirk, \emph{Lecture notes in algebraic topology},
  Graduate Studies in Mathematics, vol.~35, American Mathematical Society,
  Providence, RI, 2001.

\bibitem[Dob21]{2108.02134}
Michael~Gene Dobbins, \emph{A strong equivariant deformation retraction from
  the homeomorphism group of the projective plane to the special orthogonal
  group}, arXiv:2108.02134 (2021).

\bibitem[Fri17]{Fr}
Roberto Frigerio, \emph{Bounded cohomology of discrete groups}, Mathematical
  Surveys and Monographs, vol. 227, American Mathematical Society, Providence,
  RI, 2017.

\bibitem[GK11]{1105.0825}
\'{S}wiatos\l aw~R. Gal and Jarek K\k{e}dra, \emph{$\int_x^{hx}(g^*\alpha -
  \alpha)$}, arXiv:1105.0825 (2011).

\bibitem[HS53]{hochschild_serre53}
G.~Hochschild and J.-P. Serre, \emph{Cohomology of group extensions}, Trans.
  Amer. Math. Soc. \textbf{74} (1953), 110--134.

\bibitem[ILM06]{ismagilov_losik_michor06}
Rais~S. Ismagilov, Mark Losik, and Peter~W. Michor, \emph{A 2-cocycle on a
  symplectomorphism group}, Mosc. Math. J. \textbf{6} (2006), no.~2, 307--315,
  407.

\bibitem[Mar20]{2009.01022}
Shuhei Maruyama, \emph{The dixmier-douady class, the action homomorphism, and
  group cocycles on the symplectomorphism group}.

\bibitem[Mil83]{milnor83}
J.~Milnor, \emph{On the homology of {L}ie groups made discrete}, Comment. Math.
  Helv. \textbf{58} (1983), no.~1, 72--85.

\bibitem[NSW08]{MR2392026}
J\"{u}rgen Neukirch, Alexander Schmidt, and Kay Wingberg, \emph{Cohomology of
  number fields}, second ed., Grundlehren der mathematischen Wissenschaften
  [Fundamental Principles of Mathematical Sciences], vol. 323, Springer-Verlag,
  Berlin, 2008.

\bibitem[SS20]{savelyev_shelukhin20}
Yasha Savelyev and Egor Shelukhin, \emph{K-theoretic invariants of
  {H}amiltonian fibrations}, J. Symplectic Geom. \textbf{18} (2020), no.~1,
  251--289.

\bibitem[Yam95]{yamanoshita95}
Tsuneyo Yamanoshita, \emph{On the group of {$S^1$}-equivariant homeomorphisms
  of the {$3$}-sphere}, Publ. Res. Inst. Math. Sci. \textbf{31} (1995), no.~5,
  953--958.

\end{thebibliography}
\end{document}